\documentclass[a4paper,11pt]{article}
\usepackage[all]{xy}
\usepackage{amssymb}
\usepackage{latexsym,amsfonts}
\usepackage{amsthm}
\usepackage{mathrsfs}

\usepackage{geometry}
\newcommand{\PGL}{\mathop{\mathrm{PGL}}}
\newcommand{\Var}{\mathop{\mathrm{Var}}}
\newcommand{\Alg}{\mathop{\mathrm{Alg}}}

\newcommand{\CH}{\mathop{\mathrm{CH}}}
\newcommand{\Ch}{\mathop{\mathrm{Ch}}}
\newcommand{\Corr}{\mathop{\mathrm{Corr}}}
\newcommand{\CR}{\mathop{\mathrm{CR}}}
\newcommand{\CM}{\mathop{\mathrm{CM}}}

\newcommand{\End}{\mathop{\mathrm{End}}}
\newcommand{\Hom}{\mathop{\mathrm{Hom}}}
\newcommand{\Spec}{\mathop{\mathrm{Spec}}}
\newcommand{\coeff}{\mathop{\mathrm{coeff}}}

\newcommand{\mult}{\mathop{\mathrm{mult}}}
\newcommand{\SB}{\mathop{\mathrm{SB}}}

\newtheorem{thm}{Theorem}[section]

\newtheorem{notation}[thm]{Notation}
\newtheorem{df}[thm]{Definition}
\newtheorem{prop}[thm]{Proposition}

\newtheorem{lem}[thm]{Lemma}
\newtheorem{cor}[thm]{Corollary}

\newtheorem{conjecture}{Conjecture}
\newtheorem{question}{Question}

\begin{document}

\title{Motivic rigidity of Severi-Brauer varieties\footnote{\samepage \textit{2010 Mathematics subject classification} : 14C15, 12E15
\newline\indent
\textit{Keywords} : Grothendieck motives, upper motives, central simple algebras, Severi-Brauer varieties.}}
\author{Charles De Clercq}
\date{}
\maketitle

\begin{abstract}Let $D$ be a central division algebra over a field $F$. We study in this note the rigidity of the motivic decompositions of the Severi-Brauer varieties of $D$, with respect to the ring of coefficients and to the base field. We first show that if the ring of coefficient is a field, these decompositions only depend on its characteristic. In a second part we show that if $D$ remains division over a field extension $E/F$, the motivic decompositions of several Severi-Brauer varieties of $D$ remain the same when extending the scalars to $E$.
\end{abstract}


\section*{Introduction}

The purpose of this note is to investigate the rigidity of the motivic decompositions of projective homogeneous varieties for projective linear groups. In the first part we provide some results on the behaviour of those decompositions with respect to the ring of coefficients. We show that the motivic decomposition of a projective homogeneous variety $X$ with coefficients in $\mathbb{F}_p$ lifts to the motivic decomposition of $X$ with coefficient in any finite field of characteristic $p$. Together with the previous works of Petrov, Semenov, Zainoulline \cite{Jinv} and Vishik, Yagita \cite{vishyag}, this implies that to study the motivic decomposition of a projective homogeneous variety $X$, it is sufficient to work with $\mathbb{F}_p$-coefficients.

Another open problem is to understand the link between the motive of the Severi-Brauer varieties and the Schur index of the underlying central simple algebras. Let $D$ be a (central) division algebra over a field $F$ and $E/F$ a field extension. It is known that if $D_E$$=$$D\otimes_F E$ does not remain a division algebra, the motivic decompositions of the non-trivial Severi-Brauer varieties of $D$ ramify when extending the scalars to $E$. The converse of this assertion, which asserts that the indecomposable motives of the Severi-Brauer varieties of $D_E$ are defined over $F$, is however an open problem.

\begin{conjecture}\label{conj}Let $D$ be a division algebra over $F$ and $E/F$ be a field extension such that $D_E$ remains a division algebra. The motivic decompositions with coefficients in a finite field of the Severi-Brauer varieties of $D$ lift over $E$.
\end{conjecture}

Conjecture \ref{conj} predicts a deep relation between the motivic decompositions of Severi-Brauer varieties and the Schur index of the underlying central simple algebra, and is already known to be true for classical Severi-Brauer varieties by \cite[Corollary 2.22]{upper}. In the second part, we give a positive answer to conjecture \ref{conj}, if the reduced dimension of the underlying ideals is either squarefree or the product of $4$ and a squarefree odd number. The proofs of these results rely on a study of the indecomposable direct summands lying in the motivic decompositions of those varieties, which provides a strategy to solve conjecture \ref{conj}, namely proposition \ref{type0}. Our main tool to prove these rigidity results as well as our approach of conjecture \ref{conj} is the theory of upper motives.

\section{Preliminaries}

We fix a base field $F$, and by a variety (over $F$) we mean a smooth, projective scheme over $F$. We denote by $\Alg$/F the category of commutative $F$-algebras and by $\Var$/F the category of varieties over $F$.

\textbf{Chow groups.} Let $\Lambda$ be a commutative ring and $X$ a variety over $F$. Our basic reference for the notion of Chow groups is \cite{EKM}. We denote by $\CH(X)$ (resp. $\Ch(X)$) the integral Chow group of $X$ (resp. the Chow group of $X$ with coefficients in $\Lambda$). If $E/F$ is a field extension, consider the flat morphism $X_E$$=$$X$$\times$$ \Spec(E)$$\rightarrow $$X$. We say that an element of $\Ch(X_E)$ is $F$-rational if it lies in the image of the induced pullback $\Ch(X)\rightarrow \Ch(X_E)$, and the image of a cycle $\alpha\in \Ch(X)$ is denoted by $\alpha_E$.

\textbf{Grothendieck Chow motives.} Following \cite{EKM}, we recall briefly the construction of the category $\CM(F;\Lambda)$ of Grothendieck Chow motives with coefficients in $\Lambda$.

\begin{notation}Let $X$ and $Y$ be varieties, and consider the decomposition $X=\bigsqcup_{i=1}^{n}X_i$ of $X$ into irreducible components. The group of correspondences of degree $k$ between $X$ and $Y$ with coefficients in $\Lambda$ is defined by $\Corr_k(X,Y;\Lambda)=\bigoplus_{i=1}^n\Ch_{\dim(X_i)+k}(X_i\times Y)$.
\end{notation}

Define the category $\mathrm{C}(F;\Lambda)$ as follows. The objects of $\mathrm{C}(F;\Lambda)$ are the pairs $X(i)$, where $X$ is a variety over $F$ and $i$ an integer. A morphism $\alpha:X(i)$$\rightsquigarrow $$Y(j)$ is an element of the group $\Corr_{i-j}(X,Y;\Lambda)$, the composition being defined by \cite[Proposition 63.2]{EKM}. The category $\mathrm{C}(F;\Lambda)$ is preadditive and its additive completion, denoted by $\CR(F;\Lambda)$, is the category of correspondences with coefficients in $\Lambda$.

The category $\CM(F;\Lambda)$ of Grothendieck Chow motives with coefficients in $\Lambda$ is the pseudoabelian envelope of the category $\CR(F;\Lambda)$. Its objects are pairs $(X,\pi)$, where $X$ is an object of $\CR(F;\Lambda)$ and $\pi\in \End(X)$ is a projector. Morphisms are given by
$$\Hom((X,\pi),(Y,\rho))=\rho\circ \Hom{\!}_{\CR(F;\Lambda)}(X,Y)\circ \pi.$$

The objects of the category $\CM(F;\Lambda)$ are called motives with coefficients in $\Lambda$, or simply motives if the ring of coefficients is clear in the context. The motive of a variety $X$ is the object $M(X)$$=$$(X,\Gamma_{id_X})(0)$ of $\CM(F;\Lambda)$, where $\Gamma_{id_X}$ denotes the class of the graph of the identity in $\Ch(X\times X)$. For any integer $i$, we denote by $\Lambda(i)$ the motive $M(\Spec(F))(i)$. The set  $\{\Lambda(i),~i\in \mathbb{Z}\}$ is the set of the \emph{Tate motives}.

A morphism of commutative rings $\varphi:\Lambda\rightarrow \Lambda '$ induces a \emph{change of coefficients functor}, which is the additive functor $\coeff_{\Lambda '/\Lambda}$$:$$\CM(F;\Lambda)\rightarrow \CM(F;\Lambda ')$ being the identity on objects and acting on morphisms by $id\otimes \varphi$.

\textbf{Geometrically split motives.} For any motive $N$, the $i$-th Chow group $\Ch_i(N)$ of $N$ is the $\Lambda$-module $\Hom_{\CM(F;\Lambda)}(\Lambda(i),N)$. The motive $N$ is \emph{geometrically split}, if there is a field extension $E/F$ such that the motive $N_E$ is isomorphic to a finite direct sum of Tate motives. The field $E$ is then called a splitting field of $X$, and in this setting if $N$$=$$M(X)$ is the motive of $X$, we use the notation $\bar{X}$$=$$X_E$, the image of a cycle $\alpha$ by the pull back $\Ch(X)\rightarrow \Ch(\bar{X})$ is denoted by $\bar{\alpha}$, and the set of the $F$-rational cycles in $\Ch(\bar{X})$ is denoted by $\bar{\Ch}(X)$. The reduced endomorphism ring of a Tate twisted direct summand $N$$=$$(X,\pi)(i)$ of the motive of $X$ is defined by $\overline{\End}(N)$$=$$\bar{\pi}\circ \bar{\Ch}_{\dim(X)}(X\times X)\circ \bar{\pi}$. 

\begin{df}Assume that $E/F$ is a splitting field of a motive $N$. The \emph{dimension} of $N$ is defined by $\dim(N)=\max\{i-j,~\Ch_i(N_E)~\mbox{and}~\Ch_j(N_E)\mbox{ are not trivial}\}$. The \emph{rank} of $N$ is the dimension of the $\Lambda$-module $\Ch(N_E)$.
\end{df}

Note that the dimension and the rank of a geometrically split motive $N$ do not depend on the choice of the splitting field of $N$.

\textbf{Algebraic groups of inner type.} Let $F_{sep}/F$ be a separable closure of $F$ and $G$ be a semisimple algebraic group over $F$. Consider a maximal torus $T$ of $G$, and $\Phi(G)$ the root system associated to the split maximal torus $T_{F_{sep}}$ of $G_{F_{sep}}$.

The $\ast$-action of the absolute Galois group $\Gamma=Gal(F_{sep}/F)$ is defined as follows. First, $\Gamma$ acts on $\Phi(G)$ (see \cite[$\S$ 20]{KMRT}) and for any system of simple roots $\Pi\subset \Phi(G)$, $\sigma\cdot \Pi$ is a system of simple roots. Since the Weyl group $\mathrm{W}(G)$ acts simply transitively on the set of the simple root systems of $\Phi(G)$, $w_{\sigma}(\sigma\cdot \Pi)=\Pi$ for a uniquely determined $w_{\sigma}$. The group $\Gamma$ thus acts on $\Pi$ by $\sigma\ast \alpha=w_{\sigma}(\sigma\cdot \alpha)$. The $\ast$-action is the induced action of $\Gamma$ on the Dynkin diagram $\Delta(G)$ of $G$, which does not depend on the choice of $\Pi$.

\begin{df}A semisimple algebraic group $G$ is of \emph{inner type} if the $\ast$-action of $\Gamma$ on $\Delta(G)$ is trivial. We say otherwise that the algebraic group $G$ is of \emph{outer type}.
\end{df}

\textbf{Projective linear groups.} Let $A$ be a central simple algebra over $F$ and let $\deg(A)=\sqrt{\dim(A)}$ be its degree. In the sequel $p$ will be a prime and a central simple algebra is $p$-primary if its degree is a power of $p$. The Schur index of $A$ is the degree of a division algebra Brauer equivalent to $A$. The dimension of any right ideal $I$ of $A$ is divisible by the degree of $A$ and the quotient is the reduced dimension of $I$. The group of automorphisms of a central simple $F$-algebra $A$ is a semisimple affine algebraic group of inner type, called the \emph{projective linear group of $A$}, and denoted by $\PGL{}_1(A)$.

For any $1\leq k\leq \deg(A)$, a typical example of projective homogeneous variety for $\PGL{}_1(A)$ is given by the \emph{Severi-Brauer variety} $\SB_k(A)$ of right ideals in $A$ of reduced dimension $k$ (we refer to \cite[\S 1]{KMRT} for the basic properties of the Severi-Brauer varieties). Since the variety $\SB_k(A)$ becomes isomorphic to the Grassmann variety $G(k,\deg(A))$ over a separable closure $F_{sep}/F$, the variety $\SB_k(A)$ is of dimension $k(\deg(A)-k)$. The variety $\SB_1(A)$ is called the \emph{classical Severi-Brauer variety} of $A$.

\section{Motives of projective homogeneous varieties}

Assume that $X$ is a projective homogeneous variety for a semisimple algebraic group. As shown by \cite[Theorem 2.1]{kock} the motive of the variety $X$ is geometrically split. The extensive study of the existence and unicity of motivic decompositions $X$ is given in \cite{chermer}.

\textbf{The Krull-Schmidt theorem.} Let $C$ be a pseudoabelian category and $\mathfrak{C}$ the set of the isomorphism classes of objects of $C$. The category $C$ satisfies the Krull-Schmidt theorem if the monoid $(\mathfrak{C},\oplus)$ is free. By \cite[Corollary 35]{chermer} (see also \cite[Corollary 2.6]{upper}) the Krull-Schmidt theorem holds for the motives of projective homogeneous varieties for semisimple algebraic groups if the ring of coefficients is finite.

\textbf{Upper motives.} In this section we assume that $\Lambda$ is a finite field and $X$ is a projective homogeneous variety $X$ for a semisimple algebraic group of inner type. The first projection $p$$:$$X$$\times $$X\rightarrow X$ induces the push forward $p_{\ast}$$:$$\Ch_{\dim(X)}(X$$\times $$ X)\rightarrow \Ch_{\dim(X)}(X)=\Lambda $$\cdot$$[X]$. The \emph{multiplicity} is the morphism $\mult$$:$$\End(M(X))$$\rightarrow $$\Lambda$ such that $\mult(\alpha)$ is the element of $\Lambda$ defined by $p_{\ast}(\alpha)$$=$$\mult(\alpha)$$\cdot $$[X]$. Since the multiplicity is a morphism of rings and $\Lambda$ is a finite field, the multiplicity of a projector of $\End(M(X))$ is either equal to $0$ or to $1$.

The theory of upper motives has its origins in the study of the motives of quadrics achieved by Vishik in \cite{vish}, and was generalized in \cite{upper} to arbitrary projective homogeneous varieties. A direct summand $(X,\pi)$ of the motive of a projective homogeneous variety is called \emph{upper} if $\mult(\pi)$$=$$1$. Assuming the Krull-Schmidt theorem holds for $X$, the \emph{upper motive} of $X$, denoted by $U(X)$, is the only direct summand of $X$ which is both upper and indecomposable. Of course, the upper motive of $X$ is only defined up to isomorphism. Similarly, a direct summand $(X,\pi)$ of $X$ is \emph{lower} if $\mult(^t\pi)$$=$$1$, and we may define the \emph{lower motive} of $X$ as the only indecomposable and lower direct summand of $X$ (the transpose $^t\pi$ being the push-forward of $\pi$ with respect of the exchange isomorphism $X$$\times $$X$$\rightarrow $$X$$\times$$ X$).

Given an algebraic group $G$ of inner type, we denote by $\mathfrak{X}^{\Lambda}_G$ the set of all the indecomposable motives of all the projective $G$-homogeneous varieties in $\CM(F;\Lambda)$. By \cite[Theorem 3.5]{upper}, any element of $\mathfrak{X}^{\Lambda}_G$ is isomorphic to a Tate twist of the upper motive of a projective $G$-homogeneous variety. The theory of upper motives thus reduces the study of motivic decompositions of projective $G$-homogeneous varieties to the study of the upper motives of $G$. The present note shows how this approach can be fruitful.

\section{Upper motives and the ring of coefficients}

Recall that a ring $A$ is connected if the only idempotent elements in $A$ are $0$ and $1$. In the sequel, we only consider projective homogeneous varieties for semisimple algebraic groups of inner type. For any morphism of rings $\Lambda \rightarrow \Lambda '$, the change of coefficients functor $\coeff_{\Lambda '/\Lambda }$ is additive and maps non-trivial projectors to non-trivial projectors. In particular if $X$ is a projective homogeneous variety, the indecomposable motives of the motivic decomposition of $X$ in $\CM(F;\Lambda ')$ are a priori 'smaller pieces' than the images of the indecomposable motives of $X$ in $\CM(F;\Lambda)$ under $\coeff_{\Lambda '/\Lambda}$. Showing that the change of coefficient functor $\coeff_{\Lambda '/\Lambda}$ lifts motivic decompositions, i.e. that the indecomposable motives of $X$ in $\CM(F;\Lambda ')$ are precisely the images of the indecomposable motives of $X$ in $\CM(F;\Lambda)$ under the functor $\coeff_{\Lambda '/\Lambda}$, allows to reduce the study of motivic decomposition to simpler rings of coefficients.

Several results are achieved in this direction and motivate the main result of this section, proposition \ref{liftcoeff}. Petrov, Semenov and Zainoulline show in \cite[\S 2]{Jinv} that the functor $\coeff_{\Lambda ' /\Lambda}$ lifts motivic decompositions if the associated morphism $\Lambda\rightarrow \Lambda '$ is surjective with nilpotent kernel. By Vishik and Yagita \cite[Corollary 2.6]{vishyag}, if $\Lambda$ is a finite and connected ring, the change of coefficients functor to the residue field of $\Lambda$ also lifts motivic decompositions. These results show that in many situations the study of the motivic decompositions of $X$ is reduced to the case where the ring of coefficients is a finite field. Proposition \ref{liftcoeff} asserts that for any field $K$ of characteristic $p$, the change of coefficients functor $\coeff_{K/\mathbb{F}_p}$ lifts motivic decompositions. All these results imply that study of the motivic decompositions of projective homogeneous varieties is reduced to the study with coefficients in $\mathbb{F}_p$.

\begin{lem}\label{lemred}Assume that $N$ is a direct summand of a projective homogeneous variety $X$. The motive $N$ is indecomposable if and only if its reduced endomorphism ring is connected.
\end{lem}

\begin{proof}For any non-trivial projector $\alpha$ in $\overline{\End}(N)$, we may choose a non-trivial projector $\beta\in \End(N)$ such that $\overline{\beta}=\alpha$ by \cite[Corollary 92.5]{EKM}. In particular the endomorphism ring of $N$ is connected if and only if the reduced endomorphism ring of $N$ is connected.
\end{proof}

For the sake of completeness, we recast the proofs of \cite{dec2}. The following particular case of \cite[Corollary 1.3]{dec2} will be used to show that the image of an indecomposable direct summand of a projective homogeneous variety under the change of coefficients $\coeff_{K/\mathbb{F}_p}$ with respect to a field $K$ of characteristic $p$ is indecomposable.

\begin{lem}\label{finconn}Assume that $A$ is a finite and connected $\mathbb{F}_p$-algebra, endowed with a ring homomorphism $\varphi:A\rightarrow \mathbb{F}_p$. For any field $K$ of characteristic $p$, the tensor product $A\otimes_{\mathbb{F}_p}K$ is connected.
\end{lem}

As pointed out by T.Y. Lam, lemma \ref{finconn} holds (with a quite similar proof) if the finiteness of the ring $A$ is replaced by the fact that $A$ is Artinian. We however stick to the case where $A$ is finite since we are mainly interested in our applications to motives.

\begin{prop}\label{coeff1}Let $X$ be a projective homogeneous variety and let $K$ be a field of characteristic $p$. The image of the upper motive of $X$ in $\CM(F;\mathbb{F}_p)$ under the change of coefficients functor $\coeff_{K/\mathbb{F}_p}$ is the upper motive of $X$ in $\CM(F;K)$.
\end{prop}

\begin{proof}Setting $U(X)$ for the upper motive of $X$ in $\CM(F;\mathbb{F}_p)$, it is sufficient to show that $\coeff_{K/\mathbb{F}_p}(U(X))$ is indecomposable. The multiplicity $\mult$$:$$\End(U(X))$$\longrightarrow $$\mathbb{F}_p$ induces a morphism of rings $\overline{\mult}$$:$$\overline{\End}(U(X))$$\longrightarrow $$\mathbb{F}_p$. Since $\overline{\End}(U(X))$ is finite and connected, we may apply lemma \ref{finconn} and $\overline{\End}(\coeff_{K/\mathbb{F}_p}(U(X)))$ is connected. The motive $\coeff_{K/\mathbb{F}_p}(U(X))$ is therefore indecomposable by lemma \ref{lemred}.
\end{proof}

\begin{cor}Assume that $X$ is a projective homogeneous variety. The motive of $X$ is indecomposable in $\CM(F;\mathbb{F}_p)$ if and only if the motive of $X$ is indecomposable in $\CM(F;K)$, for any field $K$ of characteristic $p$.
\end{cor}

\begin{proof}Recall that the motive of $X$ is an indecomposable object of $\CM(F;\mathbb{F}_p)$ if and only if the upper motive of $X$ in $\CM(F;\mathbb{F}_p)$ is the whole motive of $X$. Assuming that the motive of $X$ is indecomposable in $\CM(F;\mathbb{F}_p)$, proposition \ref{coeff1} asserts that for any field $K$ of characteristic $p$, the motive $\coeff_{K/\mathbb{F}_p}(M(X))$ (which is the motive of the variety $X$ in $\CM(F;K)$) is indecomposable. The converse is clear since the functor $\coeff_{K/\mathbb{F}_p}$ is additive.
\end{proof}

\begin{prop}\label{liftcoeff}For any finite field $K$ of characteristic $p$, the functor $\coeff_{K/\mathbb{F}_p}$ lifts the motivic decompositions of projective homogeneous varieties.
\end{prop}

\begin{proof}By the Krull-Schmidt theorem it is sufficient to show that for any indecomposable motive $N$ of a projective $G$-homogeneous variety in $\CM(F;\mathbb{F}_p)$, $\coeff_{K/\mathbb{F}_p}(N)$ is indecomposable in $\CM(F;K)$. By \cite[Theorem 3.5]{upper} the motive $N$ is isomorphic to a Tate twist of the upper motive $U(Y)$ of another projective $G$-homogeneous variety. It remains to apply proposition \ref{coeff1}. 
\end{proof}

Note that the same question can be addressed for projective homogeneous varieties of outer types (i.e. when the underlying algebraic group is not of inner type). This problem is answered in \cite{dec2}.

\section{Upper motives of projective linear groups and the Schur index}

This section is dedicated to the proof of conjecture \ref{conj} for the varieties $\SB_k(D)$, where $k$ is either a squarefree number or of the form $4k'$, for a squarefree odd number $k'$. Our strategy relies on a complete description of the indecomposable summands arising in the motivic decomposition of those Severi-Brauer varieties. This qualitative study leads to a new proof of the indecomposability \cite[Theorem 4.2]{upper} of the motive of $\SB_2(D)$ in $\CM(F,\mathbb{F}_2)$, if $D$ is a $2$-primary division algebra. Let $D$ be a division algebra of degree $p^n$ over a field $F$. In the sequel we denote by $U_{k,D}$ the upper motive of the variety $\SB_{p^k}(D)$ in $\CM(F;\mathbb{F}_p)$. Note that the motives $U_{k,D}$ and $U_{k',D}$ are isomorphic if and only if $k=k'$.

\begin{df}Let $D$ be a $p$-primary division algebra over $F$ and $k$ be an integer. We say that a Severi-Brauer variety $X$ of $D$ is \emph{of type $k$} if any indecomposable direct summand of the motive of $X$ in $\CM(F;\mathbb{F}_p)$ is either isomorphic to the upper motive of $X$ or to a Tate twist of $U_{l,D}$, for some $l\leq k$.
\end{df}

The following two lemmas are direct consequences of the theory of upper motives.

\begin{lem}\label{sbktypekm1}Let $D$ be a division algebra over $F$ of degree $p^n$. For any $0\leq k\leq  n$, the Severi-Brauer variety $\SB_{p^k}(D)$ is of type $k$$-$$1$.
\end{lem}

\begin{proof}By \cite[Theorem 3.8]{upper}, any indecomposable direct summand of $\SB_{p^k}(D)$ in $\CM(F;\mathbb{F}_p)$ is isomorphic to a Tate twist of $U_{l,D}$, for some $l\leq k$. Furthermore the dimension of the motive $U_{k,D}$ is maximal by \cite[Theorem 4.1]{upper}, and in particular the motive of $\SB_{p^k}(D)$ does not contain an indecomposable direct summand isomorphic to $U_{k,D}(i)$, with $i\neq 0$. Since the Chow group $\Ch^0(\SB_{p^k}(D))$ is of rank one, there is at most one indecomposable summand of the motive of $\SB_{p^k}(D)$ which is isomorphic to $U_{k,D}$, namely the upper motive of $\SB_{p^k}(D)$.
\end{proof}

\begin{lem}\label{sbc}Let $D$ be a division algebra of degree $p^n$ (with $n\geq 1$) and $C$ be a division algebra Brauer equivalent to $D_{F(\SB_{2^{n-1}}(D))}$. For any $0$$\leq$$k$$\leq$$n$$-$$1$, the motive $(U_{k,D})_{F(\SB_{2^{n-1}}(D))}$ contains a direct summand isomorphic to the motive $U_{k,C}\oplus U_{k,C}(p^{n+k-1}(p-1))$
\end{lem}

\begin{proof}Denote by $N_{(i_1,...,i_p)}$ the motive of the variety $\SB_{i_1}(C)$$\times$$ ...$$\times$$ \SB_{i_p}(C)$, and by $F(X)$ the function field $F(\SB_{2^{n-1}}(D))$. By \cite[Theorem 10.13]{flag} the motive of $(\SB_{p^k}(D))_{F(X)}$ has a decomposition into a direct sum of twists of the motives $N_{(i_1,...,i_p)}$, where $(i_1,...,i_p)$ runs over all the non-negative integers such that $i_1$$+$$...$$+$$i_p$$=$$p^k$. The upper summand of this decomposition is $N_{(p^k,0,...,0)}$$=$$M(\SB_{p^k}(C))$, and the lower summand is $N_{(0,...,0,p^k)}(p^{n+k-1}(p-1))$, which is the motive $M(\SB_{p^k}(C))(p^{n+k-1}(p-1))$.

Recall that the dimension of the motive $(U_{k,D})_{F(X)}$ is the dimension of the variety $\SB_{p^k}(D)$ by \cite[Theorem 4.1]{upper}. The motive $(U_{k,D})_{F(X)}$ is both an upper and a lower direct summand of the motive of $(\SB_{p^k}(D))_{F(X)}$, and in particular by the Krull-Schmidt theorem the upper motive and the lower motive of $(\SB_{p^k}(D))_{F(X)}$ are both direct summands of $(U_{k,D})_{F(X)}$. Since the upper motive of $\SB_{p^k}(C)$ is also lower (again, by \cite[Theorem 4.1]{upper}) the upper motive of $(\SB_{p^k}(D))_{F(X)}$ is precisely $U_{k,C}$, and its lower motive is $U_{k,C}(p^{n+k-1}(p-1))$.
\end{proof}

\begin{notation}Considering three integers $0\leq k \leq n$ and $i$, we denote by $\mu^i_{k,n}$ the number of partitions $\lambda=(\lambda_1,...,\lambda_{p^n-p^k})$ such that $p^k\geq \lambda_1\geq ...\geq \lambda_{p^n-p^k}\geq 0$ and $|\lambda|=\sum_{j=1}^{p^n-p^k}\lambda_j$ is equal to $p^n+p^k(p^n-p^k)-i$.
\end{notation}

\begin{prop}Assume that $D$ is a division algebra over $F$ of degree $p^n$. The order of the group $\bar{\Ch}_i(\SB_1(D)$$\times $$\SB_{p^k}(D))$ is $\mu^{i+1}_{k,n}\cdot p$.
\end{prop}

\begin{proof}
Let $\mathcal{T}$ be the tautological vector bundle on $\SB_1(D)$. The product $\SB_1(D)$$\times$$\SB_{p^k}(D)$, considered as a $\SB_1(D)$-scheme via the first projection, is isomorphic to the Grassmann bundle $\Gamma_{p^n-p^k}(\mathcal{T})$ by \cite[Prop. 4.3]{karpizh}. The basis theorem \cite[Proposition 14.6.5]{fulton} then asserts that for any $i\geq 0$, there is a canonical isomorphism

$$\Ch{\!}_i(\SB{\!}_1(D)\times \SB{\!}_{p^k}(D))\simeq \bigoplus_{\lambda}\Ch{\!}_{i-p^k(p^n-p^k)+|\lambda|}(\SB{\!}_1(D))$$
where $\lambda$ runs through partitions $\lambda=(\lambda_1,...,\lambda_{p^n-p^k})$ with $p^k\geq \lambda_1\geq...\geq \lambda_{p^n-p^k}\geq 0$. 

By \cite[Proposition 2.1.1]{motifsev} the group $\bar{\Ch}^j(\SB_1(D))$ is trivial if $j>0$, hence extending the scalars to a splitting field of $E$, the only rational cycles that remain on the right side of the above isomorphism are the $0$-codimensional ones. The order of $\bar{\Ch}_i(\SB_1(D)$$\times$$\SB_{p^k}(D))$ is thus $\mu^{i+1}_{k,n}\cdot p$.
\end{proof}

\begin{cor}
\label{cor1}Assume that $D$ is a $p$-primary division $F$-algebra. If $E/F$ is a field extension such that $D_E$ is a division algebra, any $E$-rational cycle in $\Ch(\overline{\SB_1(D)\!\times \!\SB_{p^k}(D)})$ is $F$-rational.
\end{cor} 

Consider a Severi-Brauer variety $X$ of a $p$-primary division algebra $D$, and a field extension $E/F$ such that $D_E$ remains a division algebra. We now show that corollary \ref{cor1} implies that the Tate twisted motives of the classical Severi-Brauer variety of $D_E$ lying in the motivic decomposition of $X_E$ are defined over $F$.

\begin{prop}
\label{prop2}Let $D$ be a $p$-primary division algebra and $E/F$ a field extension such that $D_E$ remains division. For any $0< k\leq n$, the motive $(U_{k,D})_E$ does not contain a direct summand isomorphic to a Tate twist of the motive of $\SB_1(D_E)$.
\end{prop}

\begin{proof}Assume that $M(\SB_1(D_E))(i)$ is a direct summand of $(U_{k,D})_E$, that is to say that there are two morphisms $f:M(\SB_1(D_E))(i)$$\rightsquigarrow $$(U_{k,D})_E$ and $g:(U_{k,D})_E$$\rightsquigarrow $$M(\SB_1(D_E))(i)$ such that $g\circ f$ is the identity. Corollary \ref{cor1} asserts that $\overline{f}$ and $\overline{g}$ are $F$-rational, and thus there are two correspondences $f_1:M(\SB_1(D))(i)$$\rightsquigarrow$$ M(\SB_{p^k}(D))$ and $g_1:M(\SB_{p^k}(D))$$\rightsquigarrow $$M(\SB_1(D))(i)$ such that $\overline{g_1}=\overline{g}$ and $\overline{f_1}=\overline{f}$.

Let $\pi$ be the projector of $\End(M(\SB_{p^k}(D)))$ which defines its upper motive $U_{k,D}$. The two correspondences $\pi$$\circ$$ f_1:M(\SB_1(D))(i)$$\rightsquigarrow $$U_{k,D}$ and $g_1$$\circ$$ \pi:U_{k,D}$$\rightsquigarrow $$M(\SB_1(D))(i)$ satisfy $\overline{g_1}$$\circ $$\overline{\pi}$$\circ$$\overline{\pi}$$\circ$$\overline{f_1}=id$, and in particular by \cite[Lemma 1.2]{maksim1} an appropriate power $(g_1\circ $$\pi$$\circ$$\pi$$\circ f_1)^{\nu}$ is the identity of $M(\SB_1(D))(i)$. Setting $\tilde{g}=(g_1$$\circ$$ \pi$$\circ$$\pi$$\circ$$ f_1)^{\nu-1}$$\circ$$ g$$\circ$$\pi$ and $\tilde{f}=\pi$$\circ$$ f$, the correspondences $\tilde{g}$ and $\tilde{f}$ would define a direct summand of $U_{k,D}$ isomorphic to $M(\SB_1(D))(i)$, contradicting the indecomposability of $U_{k,D}$.
\end{proof}

\begin{prop}\label{type0}Let $X$ be a Severi-Brauer variety of a $p$-primary division algebra $D$. If $X_E$ is of type $0$ for any field extension $E/F$ such that $D_E$ is a division algebra, conjecture \ref{conj} holds for $X$.
\end{prop}

\begin{proof}We may assume by proposition \ref{liftcoeff} that the ring of coefficients is $\mathbb{F}_p$. If $D_E$ is division, the motive of $\SB_1(D)_E$ is indecomposable in $\CM(F;\mathbb{F}_p)$ by \cite[Corollary 2.22]{upper}. Since the variety $X_E$ is of type $0$, it remains to show that the upper motive of $X$ does not contain a direct summand isomorphic to a Tate twist of $M(\SB_1(D_E))$ when extending the scalars to $E$. This is precisely proposition \ref{prop2}.
\end{proof}

To the best of our knowledge, there is no example of a Severi-Brauer variety $\SB_{p^k}(D)$ of a $p$-primary division algebra which is not of type $0$. Proposition \ref{type0} therefore gives a new insight to conjecture \ref{conj}, showing that it might be reduced to the following problem.

\begin{question}\label{q1}Is any Severi-Brauer variety $\SB_{p^k}(D)$ for a $p$-primary division algebra of type $0$?
\end{question}

Question \ref{q1} has a positive answer if $k$$=$$1$ by lemma \ref{sbktypekm1}, allowing us to prove several particular cases of conjecture \ref{conj}. In the next section, we give a positive answer to question \ref{q1} if $k$$=$$p$$=$$2$, and thus prove conjecture \ref{conj} in some other cases.

\begin{thm}\label{conjproof1}
Conjecture \ref{conj} holds for any Severi-Brauer variety $\SB_k(D)$ if either the integer $k$ is squarefree or if $k$$=$$4k'$, where $k'$ is an odd and squarefree integer.
\end{thm}

\begin{proof}We first deal with the case where $k$ is squarefree. By the theory of upper motives and proposition \ref{liftcoeff}, we may assume that $D$ is $p$-primary, that the ring of coefficients is $\mathbb{F}_p$ and that $k$ equal $1$ or $p$. The result is known for $\SB_1(D)$ by \cite[Corollary 2.22]{upper} and for any field extension $E/F$ such that $D_E$ remains a division algebra, the variety $\SB_p(D_E)$ is of type $0$ by lemma \ref{sbktypekm1}. Proposition \ref{type0} thus shows that conjecture \ref{conj} holds for $\SB_p(D)$. The proof in the case where $k$ is the product of $4$ and a squarefree odd integer is given at the end of the next section.
\end{proof}

\textbf{Motivic rigidity if $p=2$.} We know provide the needed material to prove conjecture \ref{conj} for the varieties $\SB_k(D)$, where $k$ is the product of $4$ and a squarefree odd integer. The main tool is the following proposition, which asserts that for any $2$-primary division algebra, the variety $\SB_4(D)$ is of type $0$.

\begin{prop}\label{types-2}Consider a division algebra $D$ over $F$ of degree $2^n$ and an integer $0$$<$$ k$$\leq$$ n$. The Severi-Brauer variety $\SB_{2^k}(D)$ is of type $k-2$.
\end{prop} 

\begin{proof}Recall that the variety $\SB_{2^k}(D)$ is of type $k$$-$$1$ by lemma \ref{sbktypekm1}. To show that $\SB_{2^k}(D)$ is of type $k$$-$$2$, we show that there is no indecomposable factor isomorphic to a twist of $U_{k-1,D}$ in the motivic decomposition of $M(\SB_{2^k}(D))$ in $\CM(F;\mathbb{F}_2)$.

The proof goes by induction on $n$, the result being clear if $n=1$. Assume that the result is proved for every division algebra of degree $2^{n-1}$ and suppose that there is a division algebra $D$ of degree $2^n$ and a Severi-Brauer variety $\SB_{2^k}(D)$ whose motivic decomposition contains a direct summand isomorphic to a Tate twist of $U_{k-1,D}$.

We denote by $X$ the variety $\SB_{2^{n-1}}(D)$, and by $C$ the (uniquely defined up to isomorphism and of degree $2^{n-1}$) division algebra Brauer-equivalent to $D_{F(X)}$. The complete motivic decomposition of $M(\SB_{p^k}(D))_{F(X)}$ is a refinement of the following decomposition, which is given by \cite[Theorem 10.13]{flag} and was already used in the proof of lemma \ref{sbc} (here, the prime $p$ is $2$)

$$M(\SB{\!}_{2^k}(D))_{F(X)}=\bigoplus_{i+j=2^k}\left(M(\SB{\!}_i(C))\times M(\SB{\!}_j(C))\right)(i(2^{n-1}-j)).~~~~\mbox{($\ast$)}$$

By assumption, a twist of the motive $(U_{k-1,D})_{F(X)}$ is a direct summand of $M(\SB_{2^k}(D))_{F(X)}$, and in particular by lemma \ref{sbc} the motive $M(\SB_{2^k}(D))_{F(X)}$ contains a direct summand isomorphic to a twist of $U_{k-1,C}\oplus U_{k-1,C}(2^{n+k-2})$.

By \cite[Theorem 3.5]{upper}, any indecomposable direct summand of $N_{(i,j)}$$=$$M(\SB_i(C)$$\times $$\SB_j(C))$ is isomorphic to a twist of $U_{l,C}$, with $l$$\leq $$v_2(\gcd(i,j))$. The Krull-Schmidt theorem thus implies that in the decomposition $(\ast$), a twist of the motive $U_{k-1,C}$ must be a direct summand of either $N_{(2^k,0)}$, $N_{(0,2^k)}$ or $N_{(2^{k-1},2^{k-1})}$. By induction hypothesis the variety $\SB_{2^k}(C)$ is of type $k-2$, hence there is no twist of $U_{k-1,C}$ in the decomposition of $N_{(2^k,0)}$$=$$N_{(0,2^k)}$$=$$M(\SB_{2^k}(C))$. In particular a twist of the motive $U_{k-1,C}\oplus U_{k-1,C}(2^{n+k-2})$ must be a direct summand of $N_{(2^{k-1},2^{k-1})}$$=$$M(\SB_{2^{k-1}}(C)$$\times$$ \SB_{2^{k-1}}(C))$. However computing the dimension of those motives, we see that $\dim(\SB_{2^{k-1}}(C)\times \SB_{2^{k-1}}(C))=2^{n+k-1}-2^{2k-1}$ is strictly lesser than $\dim(U_{k-1,C}\oplus U_{k-1,C}(2^{n+k-2}))$$=$$2^{n+k-2}$$+$$\dim(\SB_{2^{k-1}}(C))$$=$$2^{n+k-1}$$-$$2^{2k-2}$. The variety $\SB_{2^k}(D)$ is therefore of type $k-2$.
\end{proof}

\begin{cor}[{\cite[Theorem 4.2]{upper}}]\label{indec}Assume that $D$ is a $2$-primary division algebra and $K$ is a field of characteristic $2$. The motive of $\SB_2(D)$ is an indecomposable object of $\CM(F;K)$.
\end{cor}
\begin{proof}By proposition \ref{coeff1}, it suffices to show that the motive of $M(\SB_2(D))$ is indecomposable in $\CM(F;\mathbb{F}_2)$. This directly follows from proposition \ref{types-2}.
\end{proof}

\begin{proof}[End of the proof of Theorem 4.9]
Now assume that $k$ is the product of $4$ and an odd squarefree number. If the ring of coefficients is a field of odd characteristic $p$, the proof is the same as if $k$ was squarefree. For $p=2$, we may assume that $k=4$ and $D$ is a $2$-primary division algebra by \cite[Theorem 4.1]{upper}. By proposition \ref{types-2} for any extension $E/F$ such that $D_E$ is division, the variety $\SB_4(D_E)$ is of type $0$, thus conjecture $\ref{conj}$ holds for $\SB_4(D)$ by proposition \ref{type0}.
\end{proof}

\textbf{Aknowledgements :}
I would like to express my gratitude to Nikita Karpenko for raising this question and for the helpful remarks. I also would like to thank Tsit Yuen Lam and Maksim Zhykhovich for useful discussions.

\begin{flushleft} 
\noindent Charles De Clercq\\
\texttt{email} : \texttt{declercq@math.jussieu.fr}\\
\textit{Address} : Universit\'e Paris VI, 4, place Jussieu, 75252 Paris CEDEX 5.
\end{flushleft}
\end{document}